\documentclass[a4,12pt]{article}

\usepackage{amsmath,amsthm,mathrsfs,amssymb}

\newcommand{\N}{\mathbb{N}}
\newcommand{\R}{\mathbb{R}}
\newcommand{\Z}{\mathbb{Z}}
\newcommand{\E}{\mathbb{E}}

\newcommand{\set}[1]{\left\{#1\right\}}
\newcommand{\e}{\epsilon}

\renewcommand{\P}{\mathbb{P}}
\newcommand{\supp}{\mathop{\mathrm{supp}}\nolimits}
\newcommand{\diam}{\mathop{\mathrm{diam}}\nolimits}
\renewcommand{\d}{\mathop{\mathrm{d}}\nolimits}
\newcommand{\I}{\mathcal{I}}

\newtheorem{thm}{Theorem}
\newtheorem*{thm*}{Theorem}
\newtheorem{lem}{Lemma}

\newtheorem*{prop*}{Proposition}

\theoremstyle{definition}
\newtheorem{rem}{Remark}

\newtheorem{ass}{Assumption}

\begin{document}

\title{From the Lifshitz tail to the quenched survival asymptotics in the trapping problem}
\author{Ryoki Fukushima\footnote{JSPS Research Fellow (PD), Division of Mathematics, 
Graduate School of Science, 
Kyoto University, Kyoto, 606-8502, Japan; 
E-mail: fukusima@math.kyoto-u.ac.jp} }
\date{\today}
\maketitle

\begin{abstract}
The survival problem for a diffusing particle moving among random traps is considered. 
We introduce a simple argument to derive the quenched asymptotics of the survival probability 
from the Lifshitz tail effect for the associated operator. 
In particular, the upper bound is proved in fairly general settings 
and is shown to be sharp in  the case of the Brownian motion in the Poissonian obstacles. 
As an application, we derive the quenched asymptotics for the Brownian motion in 
traps distributed according to a random perturbation of the lattice. 
\end{abstract} 

\section{Introduction and main results}\label{intro}
In this article, we consider a diffusing particle moving in random traps. 
The motion of the particle is given by a simple random walk or a Brownian motion 
and it is killed at a certain rate when it stays in a trap. 
Such a model appears in various models in chemical physics and also has some relations to 
the quantum physics in disordered media. 
We refer to the papers by Havlin and Ben-Avraham~\cite{HBA87} and den Hollander 
and Weiss~\cite{dHW94} for reviews on this model. 

The mathematical discription of the trapping model is given by the sub-Markov process with generator 
\begin{equation}
	H_{\omega} = -\kappa \Delta + V_{\omega}, \label{op}
\end{equation}
where $\Delta$ is the Laplacian on $L^2(\R^d)$ or $l^2(\Z^d)$ and $(V_{\omega}, \P)$ 
a nonnegative, stationary, and ergodic random field. Heuristically, the height of $V_{\omega}$ 
corresponds to the rate of killing. 
Let us write $(\{ X_t \}_{t \ge 0}, \{ P_x \}_{x \in \R^d\textrm{ or }\Z^d  })$ 
for the Markov process generated by $-\kappa\Delta$. 
One of the quantity of primary interest concerning this process is 
the survival probability of the particle up to a fixed time $t$, which is expressed as 
\begin{equation}
	u_{\omega}(t, x) = E_x \left[\exp\set{-\int_0^t V_{\omega}(X_s)\,ds}\right]. \label{F-K}
\end{equation}
From this expression, we can identify the survival probability as 
the Feynman-Kac representation of a solution of the initial value problem 
\begin{equation}
\begin{split}
	\partial_t u(t, x) &= \kappa \Delta u(t, x) - V_{\omega}(x) u(t, x) \quad {\rm for }
	\quad (t, x) \in (0, \infty) \times \R^d, \\
	u(0,\, \cdot\,) &\equiv 1. 
\end{split}
\end{equation}
Therefore, it is natural to expect that the long time asymptotics of the survival probability 
gives some information about the spectrum of $H_{\omega}$ around the ground state energy
and vice versa. 
This idea has been made rigorous first by Fukushima~\cite{Fuk73}, Nakao~\cite{Nak77}, and 
Pastur~\cite{Pas77} (with the analysis of some concrete examples) in the following sense: 
from the annealed long time asymptotics of the survival 
probability, one can derive the decay rate of the integrated density of states around the 
ground state energy. Their arguments are based on the fact 
that the Laplace transform of the integrated density of states can be expressed as the 
annealed survival probability for the process conditioned to come back to the starting point
at time $t$. 
Therefore, the above implication follows by an appropriate Tauberian theorem 
and, since there is the corresponding Abelian theorem (see e.g.~Kasahara~\cite{Kas78}), 
the converse is also true. 

The aim of this article is to study a relation between the \emph{quenched} asymptotics 
of $u_{\omega}(t,x)$ and 
the integrated density of states. Let us start by recalling the notion of the integrated 
density of states. To define it, we assume the following: 
\begin{ass}\label{Kato}
In the continuous setting, $V_{\omega}$ belongs to the local Kato class $K_{d, {\rm loc}}$. 
(See~\cite{CL90} or~\cite{Szn98} for the definition of $K_{d, {\rm loc}}$.)@
\end{ass}
\noindent
This assumption is sufficient to ensure that $H_{\omega}$ 
is measurable in $\omega$ as an operator. 
For the notion of measurability of operators, 
we refer to a lecture notes by Kirsch~\cite{Kir89}. 
(In fact, this is slightly stronger but we need this to utilize a uniform bound for 
the semigroup $e^{-t H_{\omega}}$ in the proof.) 
Under the above assumption, the integrated density of states of 
$H_{\omega}$ is defined as follows: 
\begin{equation}
	N^*(\lambda) = \lim_{R \to \infty} \frac{1}{(2R)^d}
	\E\bigl[\#\bigl\{ k \in \N ;\lambda^*_{\omega,\,k}\bigl((-R, R)^d\bigr) \le \lambda \bigr\}\bigr], 
	\quad *={\rm D} \textrm{ or }{\rm N},  
\end{equation}
where $\lambda^{\rm D}_{\omega,\,k}\bigl((-R, R)^d\bigr)$ 
(resp.~$\lambda^{\rm N}_{\omega,\,k}\bigl((-R, R)^d\bigr)$) is the 
$k$-th smallest eigenvalue of $H_{\omega}$ in $(-R, R)^d$ with the Dirichlet (resp.~Neumann) 
boundary condition. The existence of the limit in the right hand 
side can be proved by superadditivity (resp.~subadditivity). 

Now we state our first result. 
\begin{thm}\label{thm-upper}
Suppose that Assumptions~\ref{Kato} holds and that there exists 
a regularly varying function $\phi$ 
with index $L > 0$ such that the integrated density of states $N^{\rm D}$ associated with the operator $H_{\omega}$
in~\eqref{op} admits the upper bound 
\begin{equation}
	N^{\rm D}(\lambda) \le \exp\set{- \phi(1/\lambda) (1+o(1)) }
	\quad \mathrm{as} \quad \lambda \to 0. \label{ass-upper}
\end{equation}
Then, for any fixed $x \in \R^d$, 
\begin{equation}
	\P \textrm{-}\mathrm{a.s.}\quad 
	u_{\omega}(t, x) \le \exp\set{- t / \psi (d\log t) (1+o(1))} \quad \mathrm{as} \quad t \to \infty, 
	\label{eq-upper}
\end{equation} 
where $\psi$ is the asymptotic inverse of $\phi$. 
\end{thm}

The following assumptions are necessary only for the lower bound. 
\begin{ass}\label{moment} (\emph{Moment condition})
There exists $\alpha > 0$ such that 
\begin{equation}
	\E\biggl[\sup_{x \in [0,1]^d}\exp\{V_{\omega}(x)^{\alpha}\}\biggr] < \infty. 
\end{equation}
\end{ass}
\begin{ass}\label{correlation} (\emph{Short range correlation})
There exists $\beta > 0$ and $r_0 > 0$ such that for $\lambda > 0$ and 
boxes $A_k \subset \R^d$ or $\Z^d$ $(1 \le k \le n)$ with 
$\min_{k \neq l}{\rm dist}(A_k, A_l) > r \ge r_0$ and $\max_{1 \le k \le n} \diam (A_k) < r$, 
\begin{equation}
	\biggl| \P\biggl( \bigcap_{1 \le k \le n} E_k(\lambda) \biggr) 
	- \P(E_1(\lambda)) \P\biggl( \bigcap_{2 \le k \le n} E_k(\lambda) \biggr) \biggr| 
	< \exp \{-r^{\beta}\}, \label{ass-3}
\end{equation}
where $E_k(\lambda) = \{\lambda^{\rm N}_{\omega,\, 1}(A_k) \le \lambda\}$. 
\end{ass}

Now we are ready to state our second result. 
\begin{thm}\label{thm-lower}
Suppose that Assumptions~\ref{Kato}--~\ref{correlation} hold and 
that there exists a regularly varying function $\phi$ 
with index $L > 0$ such that the integrated density of states $N^{\rm D}$ associated with the operator 
$H_{\omega}$ in~\eqref{op} admits the lower bound 
\begin{equation}
	N^{\rm D}(\lambda) \ge \exp\set{- \phi(1/\lambda) (1+o(1)) }
	\quad \mathrm{as} \quad \lambda \to 0. \label{ass-lower}
\end{equation}
Then, there exists a constant $c_1 > 1$ such that for any fixed $x \in \R^d$,
\begin{equation}
	\P \textrm{-}\mathrm{a.s.}\quad 
	u_{\omega}(t, x) \ge \exp\set{- c_1 t / \psi (d\log t) (1+o(1))} \quad \mathrm{as} \quad t \to \infty, 
\end{equation} 
where $\psi$ is the asymptotic inverse of $\phi$. 
\end{thm}

\begin{rem}\label{Lif-exp}
The exponential behavior~\eqref{ass-upper} and~\eqref{ass-lower} of the integrated density of 
states is called the ``Lifshitz tail effect'' (cf.~\cite{Lif65}) and is typical for the trapping Hamiltonian 
$H_{\omega}$. 
The index $L$ is called ``Lifshitz exponent''. Using these terminologies, we can summarize 
our results as follows: if we have the Lifshitz tail effect with exponent $L>0$, 
then $\log u_{\omega}(t,x)$ behaves like $-t/(\log t)^{1/L+o(1)}$. 
\end{rem}

Finally we briefly comment on the relation to early studies on the quenched 
asymptotics of $u_{\omega}(t,x)$. We first give historical remarks. 
The first result in this direction has been obtained for the Brownian motion in 
the Poissonian traps by Sznitman~\cite{Szn93c} (see also~\cite{Szn98}): 
\begin{equation}
	\P \textrm{-a.s.}\quad 
	u_{\omega}(t, 0) = \exp\set{-c t/(\log t)^{2/d} (1+o(1))}
	\quad \textrm{as} \quad t \to \infty, \label{Szn}
\end{equation} 
with an explicit constant $c > 0$. 
The same asymptotics has also been proved for the discrete counterpart 
(the simple random walk in Bernoulli traps) by Antal~\cite{Ant95}. 
These results are consistent to ours since in these cases, the Lifshitz exponent 
is known to be $d/2$~\cite{Nak77, RW79}. 
Later, Biskup and K\"{o}nig~\cite{BK01b} considered the simple random walk in i.i.d.~traps 
with more general distributions. A representative example in their framework is 
\begin{equation}
	\P(V_{\omega}(0) < v) = \exp\set{-v^{-\gamma+o(1)}}
	\quad {\rm as} \quad v \to 0
\end{equation}
for some $\gamma \in (0, \infty)$. 
For such a model, they proved the quenched asymptotics 
\begin{equation}
	\P \textrm{-a.s.}\quad 
	u_{\omega}(t, 0) = \exp\set{- \tilde{\chi} t/r(t) (1+o(1))}
	\quad \textrm{as} \quad t \to \infty \label{BK}
\end{equation} 
with a constant $\tilde{\chi}>0$ described by a certain variational problem and 
a function $r(t)=(\log t)^{2/(d+2\gamma)+o(1)}$ $(t \to \infty)$ which is determined 
by a certain scaling assumption. 
It is remarkable that they also discussed the annealed asymptotics and as a consequence, 
the Lifshitz tail effect with the Lifshitz exponent $(d+2\gamma)/2$ was proved. 
Hence the relation we mentioned in Remark~\ref{Lif-exp} has already appeared in this 
special class. 

Next, we comment on some technical points. The lower bound 
(Theorem \ref{thm-lower}) is a slight modification of that of Theorem~4.5.1 
in p.196 of~\cite{Szn98} 
and not genuinely new. We include it for the completeness and to use in an application 
given in Section~\ref{PL}. 
On the other hand, the upper bound (Theorem~\ref{thm-upper}) contains some novelties. 
Besides the generality of the statement, our proof simplifies the existing arguments. 
To be more precise, our proof contains no \emph{localizing} argument which all the proofs 
of above results rely on, see e.g.~Lemma 4.6 in~\cite{BK01b}. 
We will see in Section~\ref{Poisson} that our result indeed gives a simple proof 
of the quenched asymptotics for the Brownian motion in the Poissonian obstacles. 

%
\section{Proof of the upper bound}\label{sect-upper}
We take $\kappa = 1/2$ and $x=0$ in the proof. The extension to general $\kappa $ and $x$ are verbatim. 
Also, we give the proof only for the continuous setting. The proof of the discrete case follows 
by the same argument. 
We begin with the following general upper bounds for $u_{\omega}(t,x)$ 
in terms of the principal eigenvalue. 
\begin{lem}\label{lem-upper}
Under Assumption~\ref{Kato}, there exist constants $c_2, c_3 > 0$ such that 
\begin{equation}
	u_{\omega}(t, 0) \le c_2 (1+(\lambda^{\rm D}_{\omega,\, 1}\bigl((-t, t)^d\bigr) t)^{d/2}) 
	\exp\set{ - \lambda^{\rm D}_{\omega,\, 1}\bigl((-t, t)^d\bigr) t} + e^{-c_3 t}. \label{cont-spec}
\end{equation}
\end{lem}
\begin{proof}
Let $\tau$ denote the exit time of the process from $(-t, t)^d$. 
Then, by the reflection principle, we have 
\begin{equation}
\begin{split}
	u_{\omega}(t, 0) &\le E_0 \left[\exp\set{-\int_0^t V_{\omega}(X_s)\,ds}; \tau > t\right]
	+ P_0 (\tau \le t)\\
	&\le E_0 \left[\exp\set{-\int_0^t V_{\omega}(X_s)\,ds}; \tau > t\right]
	+ e^{-c_3 t}. 
\end{split}
\end{equation}
Now,~\eqref{cont-spec} follows immediately from~(3.1.9) in p.93 of~\cite{Szn98} 
under Assumption~\ref{Kato}. 
\end{proof}
Due to this lemma, it suffices for~\eqref{eq-upper} to obtain the almost sure lower 
bound for the principal eigenvalue $\lambda^{\rm D}_{\omega,\, 1}\bigl((-t, t)^d\bigr)$. 
We use the following inequality for the integrated density of states 
\begin{equation}
\begin{split}
	N^{\rm D}(\lambda) & \ge \frac{1}{(2R)^d} 
	\E \bigl[\#\bigl\{ k \in \N ;\lambda_{\omega,\,k}\bigl((-R, R)^d\bigr) \le \lambda \bigr\}\bigr]\\
	& \ge \frac{1}{(2R)^d} \P\bigl(\lambda^{\rm D}_{\omega,\, 1}\bigl((-R, R)^d\bigr) \le \lambda\bigr), \label{IDS-lower}
\end{split}
\end{equation}
which holds for any $\lambda > 0$ and $R > 0$. 
The first inequality is an easy application of the so-called ``Dirichlet--Neumann bracketing'' 
and can be found in~\cite{CL90}, (VI.15) in p.311. 
Now, fix $\e > 0$ arbitrarily and let $\lambda = (1-\e) \psi (d\log t)^{-1}$ and $R = t$. 
Then it follows from~\eqref{IDS-lower} and~\eqref{ass-upper} that 
\begin{equation}
 \begin{split}
	\P & \left( \lambda^{\rm D}_{\omega,\, 1}\bigl((-t, t)^d\bigr) \le (1-\e) \psi (d\log t)^{-1} \right) \\
	& \le (2t)^d \exp\set{- \phi((1-\e)^{-1} \psi (d\log t)) (1+o(1))}\\
	& = 2^d t^{d-d/(1-\e)^{L}(1+o(1))}\\
	& \le t^{-\delta(\e)}
 \end{split}
\end{equation}
for some $\delta(\e) > 0$ when $t$ is sufficiently large. 
This right-hand side is summable along the sequence $t_k = e^k$ and therefore Borel-Cantelli's 
lemma shows  
\begin{equation}
	\lambda^{\rm D}_{\omega,\, 1}\bigl((-t_k, t_k)^d\bigr) \ge (1-\e) \psi (d\log t_k)^{-1}
\end{equation} 
except finitely many $k$, $\P$-almost surely. 
We can extend this bound for all large $t$ as follows: 
since $\psi (d\log t)$ is slowly varying in $t$, we have 
\begin{equation}
 \begin{split}
	\lambda^{\rm D}_{\omega,\, 1}\bigl((-t, t)^d\bigr) 
	& \ge \lambda^{\rm D}_{\omega,\, 1}\bigl((-t_k, t_k)^d\bigr) \\
	& \ge (1-\e) \psi (d\log t_k)^{-1} \\
	& \ge (1-2\e) \psi (d\log t)^{-1} 
 \end{split}
\end{equation} 
for $t_{k-1} \le t \le t_k$ when $k$ is sufficiently large. 
Combined with Lemma~\ref{lem-upper}, this proves the upper bound~\eqref{eq-upper}. 
%
\section{Proof of the Lower bound}\label{lower}
We take $\kappa = 1/2$ and $x = 0$ again. Also, we only consider the continuous case. 
As in the proof of the upper bound, the principal eigenvalue plays a key role. 
Let us write $\lambda^{\rm N}_{k}(U)$ for the $k$-th smallest eigenvalue 
of $-1/2\Delta$ in $U$ with the Neumann boundary condition. 
Then we have another inequality for the integrated density of states 
\begin{equation}
\begin{split}
	N^{\rm D}(\lambda) & \le \frac{1}{(2R)^d}
	\E \bigl[\#\bigl\{ k \in \N ;\lambda^{\rm N}_{\omega,\,k}\bigl((-R, R)^d\bigr) \le \lambda \bigr\}\bigr]\\
	& \le \frac{1}{(2R)^d} \#\bigl\{ k \in \N ;\lambda^{\rm N}_{k}((-R, R)^d) \le \lambda \bigr\}
	\P \bigl(\lambda_{\omega,\, 1}^{\rm N}\bigl((-R, R)^d\bigr) \le \lambda \bigr)\\
	& \le c_4 \P \bigl(\lambda^{\rm N}_{\omega,\, 1}\bigl((-R, R)^d\bigr) \le \lambda\bigr), \label{D-N2}
\end{split}
\end{equation}
which holds for any $\lambda \in (0,1)$ and $R > 0$. 
The first inequality can be found in~\cite{CL90} again, (VI.16) in p.~331, and 
the third one is a consequence of the classical Weyl asymptotics for the free Laplacian, 
see e.g.~Proposition~2 in Section~XIII.15 of~\cite{RS78IV}. 
For arbitrary $\e>0$, let $\lambda = (1+\e) \psi (d\log t)^{-1}$. 
Then, using~\eqref{D-N2} and~\eqref{ass-lower}, we find 
\begin{equation}
\begin{split}
	\P \bigl(\lambda^{\rm N}_{\omega,\, 1}\bigl((-R, R)^d\bigr) > \lambda \bigr)
	& \le 1- c_4^{-1} N^{\rm D} ((1+\e) \psi (d\log t)^{-1})\\
	& \le 1- c_4^{-1}(2t)^{-d/(1+\e)^L (1+o(1))}\\
	& \le 1-t^{-d+\delta(\e)}\label{C-wise}
\end{split}
\end{equation}
for some $\delta(\e)>0$ when $t$ is sufficiently large. 

Now we introduce some notations to proceed the proof. 
Let us fix a positive number 
\begin{equation}
	M > \frac{1}{\alpha} + \frac{2}{\beta} + \frac{1}{L} 
\end{equation}
and define
\begin{align}
	\I &= \bigl(-t/(\log t)^{M}, t/(\log t)^{M}\bigr)^d \cap (\log t)^{M} \Z^d, \\
	C_i &= i + \bigl(0, \psi(d \log t)^{1/2}\bigr)^d \quad (i \in \I). 
\end{align}
Note that $\min_{i \neq j} \d(C_i, C_j) > \diam (C_i)$ and both of them go to infinity as $t \to \infty$. 
Therefore, by using~\eqref{C-wise} and Assumption~\ref{correlation} recursively, we obtain 
\begin{equation}
\begin{split}
	\P & \bigl(\lambda^{\rm N}_{\omega,\, 1}(C_i) > (1+\e)\psi(d \log t)^{-1} \textrm{ for all } i \in \I \bigr)\\
	& \le \prod_{i \in \I} \P\bigl(\lambda^{\rm N}_{\omega,\, 1}(C_i) > (1+\e)\psi(d \log t)^{-1}\bigr)
	 +\exp\bigl\{ -(\log t)^2 \bigr\}\\
	& \le (1-t^{-d+\delta(\e)})^{t^d(\log t)^{-2dM}}
	 +\exp\bigl\{ -(\log t)^2 \bigr\}\\
	& \le \exp\{ -t^{\delta(\e)} (\log t)^{-2dM} \}
	 +\exp\bigl\{ -(\log t)^2 \bigr\}
\end{split}
\end{equation}
for sufficiently large $t$. Since the right hand side is summable in $t \in \N$, 
Borel-Cantelli's lemma tells us that $\P$-almost surely, 
\begin{equation}
	\textrm{there exists } i \in \I \textrm{ such that }
	\lambda^{\rm N}_{\omega,\, 1}(C_i) \le (1+\e)\psi(d \log t)^{-1}\label{N-upper}
\end{equation}
for all large $t \in \N$. 
The next lemma translates~\eqref{N-upper} to an upper bound for the Dirichlet eigenvalue:
\begin{lem}\label{lem2}
There exists a constant $c_1>1$ such that $\P$-almost surely, 
\begin{equation}
	\textrm{there exists } i \in \I \textrm{ such that }
	\lambda^{\rm D}_{\omega,\, 1}(C_i) \le c_1 \psi(d \log t)^{-1}\label{D-upper}
\end{equation}
for all large $t$. 
\end{lem}
\begin{proof}
We choose $C_i$ ($i \in \I$) for which $\lambda^{\rm N}_{\omega,\, 1}(C_i) \le (1+\e)\psi(d \log t)^{-1}$. 
This is possible for large $t \in \N$ by~\eqref{N-upper} and then it also holds for all large $t$ 
with slightly larger $\e$ by regularly varying property of $\psi$. 
Let $\phi_i^{\rm N}$ denote the $L^2$-normalized nonnegative eigenfunction corresponding to 
$\lambda^{\rm N}_{\omega,\, 1}(C_i)$ and 
$\partial_{\e} C_i$ ($i \in \I$) the set 
\begin{equation}
	\{ x \in C_i ; \d (x,\partial C_i) < \e \psi(d \log t)^{1/2} \}.
\end{equation}
We further take a nonnegative function $\rho \in C_c^1(C_i)$ which satisfies
\begin{equation}
	\rho = 1 \textrm{ on } C_i \setminus \partial_{\e} C_i \quad \textrm{and} \quad 
	\| \nabla \rho \|_{\infty} < 2 \e^{-1} \psi(d \log t)^{-1/2}.\label{rho} 
\end{equation}
Such a function can easily be constructed by a standard argument using mollifier. 
Substituting $\rho\phi_i^{\rm N}$ to the variational formula for the principal eigenvalue, 
we obtain 
\begin{equation}
	\lambda^{\rm D}_{\omega,\, 1}(C_i) \le \frac{1}{\|\rho\phi_i^{\rm N}\|_2^2}
	\int_{C_i} |\nabla (\rho\phi_i^{\rm N})|^2(x) + V_{\omega}(x)(\rho\phi_i^{\rm N})^2(x)\,dx.\label{variational}
\end{equation}
To bound the right hand side, we first use the uniform bound on eigenfunctions 
$\|\phi_i^{\rm N}\|_{\infty} \le c_5 \lambda^{\rm N}_{\omega,\, 1}(C_i)^{d/4}$
(see e.g.\ (3.1.55) in p.107 of~\cite{Szn98}) to see 
\begin{equation}
	\|\rho\phi_i^{\rm N}\|_2^2 \ge \int_{C_i \setminus \partial_{\e} C_i} \phi_i^{\rm N}(x)^2 \,dx 
	\ge 1-c_6\e.
\end{equation}
Next, it is clear from~\eqref{rho} and the above uniform bound that
\begin{equation}
\begin{split}
	\int_{C_i} & |\nabla (\rho\phi_i^{\rm N})|^2(x) + V_{\omega}(x)(\rho\phi_i^{\rm N})^2(x)\,dx\\
	& \le 2\int_{C_i} |\nabla \phi_i^{\rm N}|^2(x) + V_{\omega}(x)\phi_i^{\rm N}(x)^2\,dx 
	+ 2\int_{C_i} |\nabla \rho |^2(x) \phi_i^{\rm N}(x)^2\, dx\\
	& \le (2+8c_6\e^{-1})\psi(d \log t)^{-1}.
\end{split}
\end{equation}
Taking $\e = (2c_6)^{-1}$ and plugging these bounds into~\eqref{variational}, 
the result follows. 
\end{proof}

We also need the following almost sure upper bound. 
\begin{lem}\label{lem3}
Under Assumption~\ref{moment}, we have $\P$-almost surely, 
\begin{equation}
	\sup_{x \in (-t, t)^d}V_{\omega}(x) \le (3d \log t)^{1/\alpha} 
\end{equation} 
for sufficiently large $t$. 
\end{lem}
\begin{proof}
By Chebyshev's inequality, 
\begin{equation}
\begin{split}
	\P & \biggl(\sup_{x \in (-2t,2t)^d} V_{\omega}(x) > (3d \log t)^{1/\alpha} \biggr)\\
	& \le (4t)^d \P \biggl(\sup_{x \in [0,1]^d} V_{\omega}(x) > (3d \log t)^{1/\alpha} \biggr)\\
	& \le 4^d t^{-2d}\, \E\biggl[\sup_{x \in [0,1]^d}\exp\{V_{\omega}(x)^{\alpha}\}\biggr].
\end{split}
\end{equation}  
Since the last expression is summable in $t \in \N$, the claim follows by Borel-Cantelli's lemma 
and monotonicity of $\sup_{x \in (-t, t)^d}V_{\omega}(x)$ in $t$. 
\end{proof}

Now, we can finish the proof of the lower bound. 
We pick $\omega$ for which Lemma~\ref{lem2} and   Lemma~\ref{lem3} holds. 
Then we can find a box $C_i$ ($i \in \I$) satisfying 
\begin{equation}
	\lambda^{\rm D}_{\omega,\, 1}(C_i) \le c_1 \psi(d \log t)^{-1}\label{ev-upper}
\end{equation}
for sufficiently large $t$. 
Let $\phi_i^{\rm D}$ denote $L^2$-normalized nonnegative eigenfunction 
associated with $\lambda^{\rm D}_{\omega,\, 1}(C_i)$. 
It is easy to see that there exists a box $q+[0, 1]^d \subset C_i$ ($q \in \Z^d$) such that 
\begin{equation}
	\| \phi_i^{\rm D}\|_{\infty} \int_{q+[0, 1]^d} \phi_i^{\rm D}(x) dx 
	\ge \int_{q+[0, 1]^d} \phi_i^{\rm D}(x)^2 dx 
	\ge \frac{1}{2}\psi (d\log t)^{-d}. \label{mass}
\end{equation}
We also know the following uniform upper bound: 
\begin{equation}
	\| \phi_i^{\rm D}\|_{\infty} \le c_5 \lambda^{\rm D}_{\omega,\, 1}(C_i)^{d/4}\label{unif}
\end{equation}
from (3.1.55) in~\cite{Szn98}. 
Let us recall that the semigroup generated by $H_{\omega}$ has the 
kernel $p_{\omega}(s,x,y)$ under Assumption~\ref{Kato} (see Theorem B.7.1 in~\cite{Sim82}). 
We can bound this kernel from below by using the Dirichlet heat kernel $p_{(-t, t)^d}(s,x,y)$ 
in $(-t, t)^d$ as follows: 
\begin{equation}
\begin{split}
	p_{\omega}(s,0,y) 
	&\ge \exp\Bigl\{-s \sup_{x \in (-t, t)^d}V_{\omega}(x)\Bigr\} p_{(-t, t)^d}(s,0,y)\\
	& \ge c_7 s^{-d/2}\exp\bigl\{-s (3d \log t)^{1/\alpha} - c_8{|y|^2}/{s} \bigr\} 
	\quad \textrm{if} \quad |y|<t/2,  
\end{split}
\end{equation}
where the second inequality follows by Lemma~\ref{lem3} and a Gaussian lower bound for 
the Dirichlet heat kernel in~\cite{vdB92}. 
Taking $s=t/(\log t)^M$ and noting that $|q| < 2s$, we arrive at 
\begin{equation}
	\inf_{y \in q+[0, 1]^d} p_{\omega}(s,0,y) 
	\ge \exp\{ -c_8 s/2 \} \label{kernel}
\end{equation}
for sufficiently large $t$. 

Plugging~\eqref{ev-upper}--\eqref{kernel} into an obvious inequality, 
we arrive at
\begin{equation}
\begin{split}
	u_{\omega}(t, 0) & = \int_{\R^d} p_{\omega}(t,0,x) dx\\
	& \ge \int_{\R^d}\int_{q+[0,1]^d} p_{\omega}(s, 0, y) p_{\omega}(t-s, y, x)  
	\frac{\phi_i^{\rm D}(x)}{\|\phi_i^{\rm D} \|_{\infty}}dydx\\
	& \ge \frac{1}{\|\phi_i^{\rm D}\|_{\infty}} \exp\bigl\{ -\lambda^{\rm D}_{\omega,\, 1}(C_i)t-c_8 s/2 \bigr\}
	\int_{q+[0, 1]^d} \phi_i^{\rm D}(x) dx\\
	& \ge c_9 \psi (\log t)^{-3d/2} \exp\{ -c_1 t/\psi(d \log t)-c_8 s/2 \}, 
\end{split}
\end{equation}
where in the third line, we have replaced $p_{\omega}$ by the kernel of the semigroup 
generated by $H_{\omega}$ with the Dirichlet boundary condition outside $C_i$. 
This completes the proof of the lower bound of Theorem~\ref{thm-lower} 
since $s=t/(\log t)^{M}$ was chosen to be $o(t/\psi(\log t))$. 

\section{Examples}\label{examples}
We apply our results to two models in this section. 
The first is the Brownian motion in the Poissonian obstacles, where we see that our 
result recovers the correct upper bound. 
The second is the Brownian motion in a perturbed lattice traps introduced in~\cite{F3}, 
for which the quenched result is new. 

\subsection{Poissonian obstacles}\label{Poisson}
Let us consider the standard Brownian motion ($\kappa = 1/2$) killed by 
the random potential of the form
\begin{equation}
	V_{\omega}(x) = \sum_{i} W(x-\omega_i), 
\end{equation}
where $(\omega=\sum_i \delta_{\omega_i}, \P_{\nu})$ is a Poisson point process 
with intensity $\nu > 0$ and $W$ is a nonnegative, bounded, and compactly supported function. 
As is mentioned in Section~\ref{intro}, Sznitman proved in~\cite{Szn93c} 
the quenched asymptotics for this model: 
\begin{equation}
	\P_{\nu} \textrm{-a.s.}\quad 
	u_{\omega}(t, 0) = \exp\set{-c(d,\nu) t/(\log t)^{2/d} (1+o(1))}
	\quad \textrm{as} \quad t \to \infty, \label{Poi-quench}
\end{equation} 
where $c(d,\nu) = \lambda_d  (\nu\omega_d / d)^{2/d} $ with $\lambda_d$ denoting the principal 
Dirichlet eigenvalue of $-1/2\Delta$ in $B(0,1 )$ and $\omega_d=| B(0,1) |$. 

We can recover the upper bound by using classical Donsker-Varadhan's result~\cite{DV75c} 
and Theorem~\ref{thm-upper}. 
Indeed, the above potential clearly satisfies Assumption~\ref{Kato} and 
the asymptotics of the integrated density of states 
\begin{equation}
	\log N^{\rm D}(\lambda) \sim -\nu \omega_d \lambda_d^{d/2}
	\lambda^{-d/2} \quad \textrm{as} \quad \lambda \to 0 \label{Lif-Poi}
\end{equation}
has been derived by Nakao~\cite{Nak77} by applying an exponential Tauberian theorem 
to Donsker-Varadhan's asymptotics 
\begin{equation}
	\E[u_{\omega}(t, 0)] = \exp\set{-\tilde{c}(d,\nu) t^{\frac{d}{d+2}} (1+o(1))} 
	\quad \textrm{as} \quad t \to \infty 
\end{equation}
with 
\begin{equation}
	\tilde{c}(d,\nu)=\frac{d+2}{2}(\nu \omega_d)^{\frac{2}{d+2}}
	\Bigl(\frac{2\lambda_d}{d}\Bigr)^{\frac{d}{d+2}}.
\end{equation}
Now an easy computation shows that the asymptotic inverse of the right hand side 
of~\eqref{Lif-Poi} is 
\begin{equation}
	\psi(\lambda) = \lambda_d ^{-1} (\nu\omega_d)^{-2/d} \lambda^{2/d}
\end{equation}
and then Theorem~\ref{thm-upper} proves the upper bound in~\eqref{Poi-quench}. 

\begin{rem}
	In this case, the lower bound given by Theorem~\ref{thm-lower} is not sharp 
	as is obvious from the statement. (In the proof, we lose the precision in Lemma~\ref{lem2}.) 
	However, the lower bound can be complemented by a rather direct and simple argument 
	in the Poissonian soft obstacles case, see~\cite{Szn93c}. 
	So our argument simplifies the harder part. 
\end{rem}

\subsection{Perturbed lattice traps}\label{PL}
In this subsection, we use our results to derive the quenched asymptotics for 
the model introduced in~\cite{F3}. 
We consider the standard Brownian motion ($\kappa = 1/2$) killed by the potential 
of the form 
\begin{equation}
	V_{\omega}(x) = \sum_{q \in \Z^d} W(x-q-\omega_q), 
\end{equation}
where $(\{\omega_q\}_{q \in\Z^d}, \P_{\theta})$ ($\theta > 0$) is a collection of 
independent and identically distributed random vectors with density 
\begin{equation}
	\P_{\theta}(\omega_q \in dx) = N(d, \theta)\exp\bigl\{ -|x|^{\theta} \bigr\}dx 
\end{equation}
and $W$ is a nonnegative, bounded, and compactly supported function. 
The author has derived the annealed asymptotics for this model in~\cite{F3} and 
also proved the following Lifshitz tail effect as a corollary: 
 \begin{equation*}
  \begin{split}
   \log N^{\rm D}(\lambda) \asymp _{\lambda \to 0}
   \left\{
   \begin{array}{lr}
    -\lambda^{-1-\frac{\theta}{2}}\left( \log \frac{1}{\lambda} \right)^
    {-\frac{\theta}{2}} &(d=2),\\[8pt]
    -\lambda^{-\frac{d}{2}-\frac{\theta}{d}}  &(d \ge 3),
   \end{array}\right.
  \end{split}
 \end{equation*}
where $f(x) \asymp_{x \to *} g(x)$ means 
$0<\liminf_{x \to *}f(x)/g(x)\le\limsup_{x \to *}f(x)/g(x)<\infty$. 

We can prove the quenched asymptotics from this result. 
\begin{thm}
For any $\theta > 0$ and $x \in \R^d$, we have 
\begin{equation}
 \begin{split}
	\log u_{\xi}(t, x) \asymp_{t \to \infty} \left\{
	\begin{array}{lr}
	- t \, (\log t)^{-\frac{2}{2+\theta}}(\log \log t)^{-\frac{\theta}{2+\theta}} &\quad (d=2),\\[8pt]
	-t \, (\log t)^{-\frac{2d}{d^2+2\theta}}  &\quad (d \ge 3),
	\end{array}\right.\label{Lif-PL}
 \end{split}
\end{equation} 
with $\P_{\theta}$-probability one. 
\end{thm}
\begin{proof}
The Assumption~\ref{Kato} is clearly satisfied since $V_{\omega}$ is locally bounded 
almost surely. Hence the upper bound readily follows by computing the asymptotic 
inverse of~\eqref{Lif-PL} and using Theorem~\ref{thm-upper}. 
To use Theorem~\ref{thm-lower}, we have to verify Assumptions~\ref{moment} and~\ref{correlation}. 
The former is rather easy and can be found in Lemma~11 in~\cite{FU09}. 
The latter is verified as follows: we first fix $r_0 > 0$ sufficiently large so that 
$\supp W \subset B(0, r_0/4)$. For $r>r_0$ and boxes $\{A_k\}_{1 \le k \le n}$ as in 
Assumption~\ref{correlation}, let us define events 
\begin{align}
	E_1 & \stackrel{\rm def}{=}
	\bigl\{\textrm{for all } q \in \Z^d \textrm{ with }\d(q, A_1) \le r/2, 
	\d(q+\omega_q, A_1) \le 3r/4\bigr\}, \\
	E_2 & \stackrel{\rm def}{=}
	\bigl\{\textrm{for all } q \in \Z^d \textrm{ with }\d(q, A_1) \ge r/2, 
	\d(q+\omega_q, A_1) \ge r/4 \bigr\}. 
\end{align}
Then, $\lambda^{\rm N}_{\omega,\, 1}(A_1) $ and $\{\lambda^{\rm N}_{\omega,\, 1}(A_k) \}_{2 \le k \le n} $ 
are mutually independent  on $E_1 \cap E_2$ thanks to our choice of $r_0$. 
Therefore, the left hand side of~\eqref{ass-3} is bounded by $\P_{\theta}(E_1^c)+\P_{\theta}(E_2^c)$. 
Let us denote the $s$-neighborhood of $A_1$ by $N_s(A_1)$. 
The first term is estimated as 
\begin{equation}
\begin{split}
	\P_{\theta}(E_1^c) &\le \P_{\theta}\bigl(|\omega_q| \ge r/4
	\textrm{ for some } q \in \Z^d \cap N_{r/2}(A_1)\bigr)\\
	& \le N(d, \theta) \# \bigl\{q \in \Z^d \cap N_{r/2}(A_1)\bigr\} 
	\int_{|x| \ge r/4}\exp\bigl\{ -|x|^{\theta }\bigr\} dx\\
	& \le N(d, \theta) r^d \exp\bigl\{ -(r/8)^{\theta}\bigr\} 
\end{split}
\end{equation}
for large $r$, where we have used $\diam (A_1) < r$ in the last line. 
Next, we bound the second term $\P_{\theta}(E_2^c)$. 
Using the distribution of $\omega_q$, we have 
\begin{equation}
\begin{split}
	\P_{\theta}(E_2^c) &= \P_{\theta}\bigl(q+\omega_q \in N_{r/4}(A_1)
	\textrm{ for some } q \in \Z^d \setminus N_{r/2}(A_1)\bigr)\\
	& \le N(d, \theta) \sum_{q \in \Z^d \setminus N_{r/2}(A_1)} 
	\int_{N_{r/4}(A_1)}\exp\bigl\{ -|x-q|^{\theta }\bigr\} dx\\
	& \le N(d, \theta) r^d \sum_{q \in \Z^d \setminus N_{r/2}(A_1)}
	\exp\bigl\{ -\d (q, N_{r/4}(A_1))^{\theta}\bigr\}. 
\end{split}
\end{equation}
We can assume by shift invariance that $A_1$ is centered at the origin. 
We divide the sum into two parts $\{|q| \le r\}$ and $\{|q| > r\}$. 
The former part of the sum is bounded by 
\begin{equation}
\begin{split}
	\# &\bigl\{q \in \Z^d \cap B(0,r)\bigr\} \sup_{q \in \Z^d \setminus N_{r/2}(A_1)} 
	\exp\bigl\{ -\d (q, N_{r/4}(A_1))^{\theta}\bigr\}\\ 
	&\le c_{10}r^d \exp\bigl\{-(r/4)^{\theta}\bigr\}. 
\end{split}
\end{equation}
For the latter part, we use the fact that $N_{r/4}(A_1) \subset B(0, 3r/4)$, 
which follows from the assumption $\diam (A_1)<r$. 
By using this fact, we find 
\begin{equation}
	\d (q, N_{r/4}(A_1)) \ge |q|-3r/4 > |q|/4 
	\quad \textrm{for} \quad |q| > r 
\end{equation}
and therefore 
\begin{equation}
	\sum_{q \in \Z^d \setminus N_{r/2}(A_1), \,|q|>r} \exp\bigl\{ -\d (q, N_{r/4}(A_1))^{\theta}\bigr\}
	\le \sum _{q \in \Z^d, \,|q| > r} \exp\bigl\{ -|q/4|^{\theta}\bigr\}. 
\end{equation}
It is not difficult to see that this right hand side is bounded by $\exp\{ -(r/8)^{\theta}\}$ 
for sufficiently large $r$. Combining all the estimates, we arrive at 
\begin{equation}
	\P_{\theta}(E_1^c) + \P_{\theta}(E_2^c) 
	\le N(d, \theta) r^d \bigl(2+c_{10}r^d\bigr)\exp\bigl\{ -(r/8)^{\theta}\bigr\}
\end{equation}
for large $r$, which verifies Assumption~\ref{correlation}. 
\end{proof}

\section*{Acknowledgement} 
The author is grateful to professor Naomasa Ueki for useful conversations 
on the theory of random Schr\"{o}dinger operators. 
He also would like to thank professor Wolfgang K\"{o}nig for helpful discussions. 

\newcommand{\noop}[1]{}

\end{document}